\documentclass[10pt, A4paper]{amsart}

\usepackage{amsmath,amsthm,amssymb}
\usepackage[dvips]{graphicx}
\usepackage{color}
\newtheorem{theorem}{Theorem}[section]
\newtheorem{proposition}[theorem]{Proposition}
\newtheorem{lemma}[theorem]{Lemma}
\newtheorem{corollary}[theorem]{Corollary}

\theoremstyle{definition}
\newtheorem{remark}[theorem]{Remark}

\newcommand{\B}{\mathcal{B}}
\newcommand{\F}{\mathcal{F}}
\newcommand{\W}{\mathcal{W}}

\newcommand{\gene}{\mathcal{L}}
\newcommand{\DD}{\mathbb{D}}
\newcommand{\R}{\mathbf{R}}

\DeclareMathOperator{\esssup}{ess\sup}

\begin{document}

\title
[A new proof for the convergence of Picard's filter]
{A new proof for the convergence of Picard's filter
using partial Malliavin calculus}

\author{Hideyuki Tanaka}
\date{May.21, 2014}
\address{
Department of Mathematical Sciences, Ritsumeikan University, 1-1-1 Nojihigashi, Kusatsu, Shiga
525-8577, Japan 
}
\email{hitanaka@fc.ritsumei.ac.jp}
\maketitle

\begin{abstract}
The discrete-time approximation for nonlinear filtering problems is related to 
both of strong and weak approximations of stochastic differential equations. 
In this paper, we propose a new method of proof 
for the convergence of approximate nonlinear filter analyzed by Jean Picard (1984), 
and show a more general result than the original one. 
For the proof, we develop an analysis of Hilbert space valued  functionals on Wiener space. 
\end{abstract}

\section{Introduction}
The aim of this paper is to determine the convergence rate of Picard's filter for nonlinear filtering 
in a more general condition than that of Picard (\cite{P84}), 
and to understand deeply why the scheme can perform with the rate. 
Although Picard's filter is based on an Euler-type approximation of stochastic differential equations, 
the error estimate does not rely on the standard argument of strong and weak convergence of the Euler-type scheme. 
As seen in the following, the properties of stochastic integrals under a conditional probability 
make the proof of convergence much more complicated. 

Let us first formulate the nonlinear filtering problem with continuous time observations. 
Consider a stochastic process $(X_t)_{t \geq 0}$ (often called the signal process) defined as the solution of 
an $N$-dimensional stochastic differential equation 
\begin{eqnarray}\label{eq:sde}
X_t = x + \int_0^t b(X_s)ds + \int_0^t\sigma(X_s)dB_s
\end{eqnarray} 
with $x \in \R^N$ and an $N$-dimensional standard Brownian motion $B = (B_t)_{t \geq 0}$ 
on a probability space $(\Omega, \F, P)$ 
with a filtration $(\F_t)_{t \geq 0}$ satisfying the usual conditions. 
We observe another $d$-dimensional process $(Y_t)_{t \geq 0}$ (called the observation process) defined by
\[
Y_t = \int_0^t h(X_s)ds + W_t
\]
where $W = (W_t)_{t \geq 0}$ is a $d$-dimensional standard Brownian motion independent of $B$. 
We denote the filtrations associated to $B$ and $Y$ with $P$-null sets by $(\F_t^B)$ and $(\F_t^Y)$ respectively. 
The primary goal of nonlinear filtering problem is 
to investigate the evolution of the conditional distribution of $X_T$ under the observation $(Y_t)_{0\leq t \leq T}$. 
In other words, we are interested in computing the value 
\begin{equation}\label{eq:target}
E^P[g(X_T)|\F_T^Y].
\end{equation}
For this purpose, we consider the new probability measure $Q$ on $\F_\infty = \sigma(\cup_{t \geq 0}\F_t)$ 
under which $(Y_t)$ is a standard Brownian motion independent of $(X_t)$, 
and $(X_t)$ has the same law under $P$ and $Q$. 
Throughout the paper, we denote the expectation under $Q$ by $E[ \ \cdot \ ]$. 
Then the conditional expectation (\ref{eq:target}) has the expression 
\[
E^P[g(X_T)|\F_T^Y] = \frac{E[g(X_T)\Phi_T|\F_T^Y]}{E[\Phi_T|\F_T^Y] } 
\]
with the Radon-Nikodym derivative 
\begin{align*}
\Phi_t = \exp\Big( \sum_{j=1}^d \Big(\int_0^t h^j(X_s)dY^j_s -\frac{1}{2}\int_0^t (h^j)^2(X_s)ds \Big) \Big). 
\end{align*}
This is called the Kallianpur-Striebel formula (cf.\ \cite{KaSt68}, \cite{BC09}). 
We need time discretization methods in order to compute $E[g(X_T)\Phi_T|\F_T^Y]$ 
since the stochastic integral term cannot be computed exactly. 

In what follows, we discuss a discrete-time approximation scheme for $\Phi_t$ 
under the probability measure $Q$. 
Let us use the notation 
$\|\cdot\|_p := E[|\cdot|^p]^{1/p}$. 
Fix $T>0$ and $\eta(t) = t_i := iT/n$ if $t \in [iT/n, (i+1)T/n)$. 
We now consider an approximation by a Riemann sum for $\Phi_T$ as
\begin{align*}
\tilde{\Phi}_t \equiv \tilde{\Phi}_t(X) := \exp\Big( \sum_{j=1}^d \Big(\int_0^t h^j(X_{\eta(s)})dY^j_s 
- \frac{1}{2}\int_0^t (h^j)^2(X_{\eta(s)})ds \Big) \Big).
\end{align*}
Jean Picard showed the following surprizing result of $L^2$-convergence of $E[g(X_T)\tilde{\Phi}_T|\F_T^Y]$. 
\begin{theorem}[\cite{P84}] 
Assume that $g$, $b$ and $\sigma$ are Lipshitz continuous and $h \in C_b^2(\R^N;\R^d)$. 
Then
\begin{equation}\label{eq:picard}
\Big\| E[g(X_T)\Phi_T|\F_T^Y] 
- E[g(X_T)\tilde{\Phi}_T|\F_T^Y]  \Big\|_2 \leq \frac{C_T}{n}.
\end{equation}
\end{theorem}

\begin{remark}
The assumption $\|h\|_\infty < \infty$ can be weakened (see \cite{P84}, \cite{C11}). 
For example, Picard (\cite{P84}) discusses the condition 
\[
E\Big[ \exp\Big( (1+\varepsilon) T H\Big( \sup_{0 \leq t \leq T} |X_t| \Big) \Big)\Big] < \infty, \ \mbox{ for some } \varepsilon >0 
\]
where
\[
H(y) := \sup\Big\{ \sum_{j=1}^d (h^j)^2(x) ; |x| \leq y \Big\}.
\]
\end{remark}

The convergence error (\ref{eq:picard}) is related to both of 
weak convergence of $\F_T^B$- measurable random variables 
and strong convergence of $\F_T^Y$-measurable random variables. 
Very roughly speaking, the order of convergence of the error is mainly from 
\[
\int_0^T (h(X_s)-h(X_{\eta(s)}))dY_s.
\] 
We notice that the difference $h(X_s)-h(X_{\eta(s)})$ has the weak error of $O(1/n)$, 
but this is averaged over the trajectory of $(Y_s)$. 
That is why the rate of convergence is not so obvious.
The proof given by Picard is quite complicated since we have to deal carefully with $\int_0^T \cdot \ dY_s$ 
under the conditional expectation $E[\ \cdot \ |\F_T^Y]$.
In this work, we generalize the result (\ref{eq:picard}) 
in terms of the regularity of $g$ (without any ellipticity condition) 
and $L^p$-estimates with $p>2$ 
using several techniques in Malliavin calculus, 
and however, $h$ is basically assumed to be bounded because of the difficulty 
in $L^p$-moment estimates for $\Phi_T$ and $\tilde{\Phi}_T$. 
See the main result in Theorem \ref{thm:main} and its proof. 

We review here numerical methods required for the simulation of Picard's filter 
$
E[g(X_T)\tilde{\Phi}_T|\F_T^Y].
$ 
Except in some specific situations the closed-form distribution of $X_t$ is not available, and therefore we need 
some time discretization schemes applied to $X_t$. 
Let $\tilde{X}$ be a time discretization scheme for $X$, 
such as the Euler-Maruyama approximation or the stochastic ODE approximations via cubature formulas on Wiener space 
(\cite{LV04}, \cite{NV08}). 
Then we have to analyze the error 
\begin{equation}\label{eq:secondstep}
\Big\| E[g(X_T)\tilde{\Phi}_T(X) |\F_T^Y] - E[g(\tilde{X}_T)\tilde{\Phi}_T(\tilde{X}) |\F_T^Y]  \Big\|_p
\end{equation}
and this type of problem is discussed in e.g.\ \cite{CG07}, \cite{CO}. 
In the case where $\tilde{X}$ is Euler-Maruyama scheme, 
several researchers give error estimates for (\ref{eq:picard}) and (\ref{eq:secondstep}) simultaneously (e.g.\ \cite{Ta86}, \cite{MT09}). 
We additionally have to discuss the simulation of $E[g(\tilde{X}_T)\tilde{\Phi}_T(\tilde{X})|\F_T^Y]$ 
via the Monte Carlo method. 
In practice, the procedure of estimation for this is performed step-by-step for each 
observation time $T = t_1, \cdots, t_n, \cdots$. 
Hence it is important to construct special simulation methods, recursively in time $T$, 
to avoid the recalculation of the conditional expectation and explosion of time series data. 
For the reason, particle filters (or sequencial Monte Carlo methods) were originally developed 
by \cite{Gor93} and \cite{Ki93} for discrete time filtering. 
Recent developments of particle filters can be found in \cite{DJ11} and references therein. 

Another approach to the computational problem for (\ref{eq:target}) is known as 
the stochastic partial differential equation (SPDE) approach.
We can derive the equation of the dynamics of $t \mapsto E[g(X_t)\Phi_t|\F_t^Y]$ ($g \in C^2$) 
which is called the Zakai equation (cf.\ \cite{BC09}, \cite{K11}). 
The Zakai equation follows a SPDE with the finite dimensional noise $Y$. 
In that case, we have to consider time discretizations for the SPDE and 
give some error estimates for strong convergence (see  e.g.\ \cite{G06}). 
We point out the relationship between the Zakai equation 
and Picard's filter $E[g(X_t)\tilde{\Phi}_t|\F_t^Y]$ in Remark \ref{rem:zakai}. 

This paper is organized as follows. 
In Section 2, we state the main result 
which is an extension of Picard's theorem, and shall give only the outline of the proof. 
In Section 3, we show the main part of the proof 
using infinite dimensional analysis on Wiener space, 
and in Section 4 we give some remarks on this research. 

\section{The Main result}
\subsection{An extension of Picard's theorem}
Let us fix $T > 0$. Throughout the paper, the condition 
\begin{equation}\label{eq:martingale_property}
E^P[\Phi_T^{-1}] = 1
\end{equation}
is always assumed to define the probability measure $Q$ on $\F_T$, 
i.e. $Q(A) := E^P[1_A \Phi_T^{-1}]$ for $A \in \F_T$. 
The assumptions \rm{(A2)}-\rm{(A3)} introduced below 
imply the condition (\ref{eq:martingale_property}). 
See Kallianpur \cite{K80}, Section 11.3. 

We shall extend Picard's theorem as follows. 
\begin{theorem}\label{thm:main}
Assume that the following conditions hold: 
\begin{itemize}
\item[\rm{(A1)}] The function $g: \R^N \rightarrow \R$ is a measurable function 
such that $g(X_T) \in \cap_{p \geq 1}L^p(\Omega, \F_T, Q)$. 
\item[\rm{(A2)}] The coefficients $b$ and $\sigma$ are Lipshitz continuous.
\item[\rm{(A3)}] The function $h:\R^N \rightarrow \R^d$ is a $C^2$-function 
of polynomial growth with all derivatives.
\item[\rm{(A4)}] For every $p \geq 1$, 
\[
\|\Phi_T\|_p + \sup_{n}\|\tilde{\Phi}_T\|_p \leq K(p,T) < \infty
.\]
\end{itemize}
Then for every $p \geq 1$, there exists a constant $C = C(p,T) > 0$ such that 
\begin{equation}\label{eq:main}
\Big\| E[g(X_T)\Phi_T|\F_T^Y] 
- E[g(X_T)\tilde{\Phi}_T|\F_T^Y]  \Big\|_p \leq \frac{C}{n}.
\end{equation}
\end{theorem}

A typical example of \rm{(A4)} is that $h$ is bounded. 
The following corollary for the convergence of the normalized conditional expectation  
is an immediate consequence of Theorem \ref{thm:main}.
\begin{corollary}
Suppose the assumptions \rm{(A1)}-\rm{(A3)} hold, and moreover $h$ is assumed to be bounded. 
Then for every $p \geq 1$, there exists a constant $C = C(p,T) > 0$ such that 
\begin{equation*}
E^P\Big[ \Big|E^P[g(X_T)|\F_T^Y]
- \frac{E[g(X_T)\tilde{\Phi}_T|\F_T^Y]}{E[\tilde{\Phi}_T|\F_T^Y]}\Big|^p \Big]^{1/p} \leq \frac{C}{n}.
\end{equation*}
\end{corollary}
\begin{proof}
Let $\rho_T(g) := E[g(X_T)\Phi_T|\F_T^Y]$ and $\tilde{\rho}_T(g) := E[g(X_T)\tilde{\Phi}_T|\F_T^Y]$. 
The error is expressed as 
\begin{align*}
\frac{\rho_T(g)}{\rho_T({\bf 1})} - \frac{\tilde{\rho}_T(g)}{\tilde{\rho}_T({\bf 1})} = 
\frac{\rho_T(g) - \tilde{\rho}_T(g)}{\rho_T({\bf 1})} + \frac{\tilde{\rho}_T(g)}{\rho_T({\bf 1})\tilde{\rho}_T({\bf 1})}(\tilde{\rho}_T({\bf 1}) - \rho_T({\bf 1})).
\end{align*}
It is possible to show from the boundedness of $h$ that the $L^p(\Omega, \F_T, Q)$-norms of 
$\Phi_T$, $\tilde{\Phi}_T$, $\rho_T({\bf 1})^{-1}$ 
and $\tilde{\rho}_T({\bf 1})^{-1}$ are bounded for every $p\geq 1$. 
Hence we obtain from Cauchy-Schwarz's inequality 
\begin{align*}
E^P\Big[\Big|\frac{\rho_T(g)}{\rho_T({\bf 1})} - \frac{\tilde{\rho}_T(g)}{\tilde{\rho}_T({\bf 1})}\Big|^p\Big]^{1/p}
& =  E\Big[ \Big| \frac{\rho_T(g)}{\rho_T({\bf 1})} - \frac{\tilde{\rho}_T(g)}{\tilde{\rho}_T({\bf 1})}\Big|^p \Phi_T \Big]^{1/p}
\\ & \leq C_1(p,T) \| \rho_T(g) - \tilde{\rho}_T(g) \|_{2p} + C_2(p,T) \| \rho_T({\bf 1}) - \tilde{\rho}_T({\bf 1}) \|_{2p}, 
\end{align*}
which proves the desired result. 
\end{proof}

\begin{remark}
For the proof of Theorem \ref{thm:main}, 
the probability space $(\Omega, \F_T, Q)$ can be replaced by any other probability space 
on which $(X_t,Y_t)_{0\leq t \leq T}$ has the same law. 
In the following, we fix the probability space  
so that $(B_t)_{0\leq t \leq T}$ and $(Y)_{0\leq t \leq T}$ are independent Brownian motions, and $(X_t)_{0\leq t \leq T}$ is the solution of (\ref{eq:sde}). 
The probability space will be assumed to be the Wiener space in Section \ref{sec:3}. 
\end{remark}

\begin{remark}\label{rem:zakai}
As mentioned in the introduction, the time evolution $\rho_t(g): t \mapsto E[g(X_t)\Phi_t|\F_t^Y], (g \in C_b^2)$ solves the Zakai equation 
\[
\rho_t(g) = \rho_0(g) + \int_0^t \rho_s(\gene g) ds + \int_0^t \rho_s( g h^\mathsf{T}) dY_s
\]
where $\rho_0(g) = E[g(X_0)] = g(x)$ and $\gene$ is the generator of $X$, i.e. 
\[
(\gene g)(x) = \sum_{i=1}^N b^i(x)\frac{\partial g}{\partial x_i}(x)
+ \frac{1}{2} \sum_{i, j=1}^N (\sigma^i\sigma^j)(x)\frac{\partial^2 g}{\partial x_i \partial x_j}(x).
\]
Picard's filter $\tilde{\rho}_t(g) : t \mapsto E[g(X_t)\tilde{\Phi}_T | \F_t^Y]$ can be understood as 
a semigroup-type approximation (or Markov chain approximation) in the following sense. 
Let $X_t^x$ be a stochastic flow of the SDE (\ref{eq:sde}) and $(P_tg)(x) := E[g(X_t^x)]$. 
Define a parameterized operator $\tilde{P}_t^y$, $y \in \R^d$ by 
\[
(\tilde{P}_t^y g)(x) := (P_tg)(x) 
\exp\Big(\sum_{j=1}^d \Big(h^j(x)y^{j} 
- \frac{1}{2}(h^j)^2(x)t \Big)\Big).
\]
Then we can deduce that for $t_i \leq t < t_{i+1}$, 
\[
\tilde{\rho}_t(g) = \tilde{P}_{t_1-t_0}^{Y_{t_1}-Y_{t_0}} 
\circ \cdots \circ \tilde{P}_{t_{i}-t_{i-1}}^{Y_{t_i}-Y_{t_{i-1}}} \circ \tilde{P}_{t-t_i}^{Y_{t}-Y_{t_i}}(g),  
\]
and $\tilde{P}_{t-t_i}^{Y_{t}-Y_{t_i}}(g)(x)$ is a solution of the evolution equation 
\[
\tilde{P}_{t-t_i}^{Y_{t}-Y_{t_i}}(g) = 
g(x) + \int_{t_i}^t \tilde{P}_{s-t_i}^{Y_{s}-Y_{t_i}}(\gene g) ds 
+ \int_{t_i}^t \tilde{P}_{s-t_i}^{Y_{s}-Y_{t_i}}( g ) h^\mathsf{T}(x) dY_s, 
\]
which can be considered as the Zakai equation with the freezing coefficient $h(x)$. 
\end{remark}

\subsection{Outline of proof}
The proof of Theorem \ref{thm:main} is entirely different from the original one in \cite{P84}. 
Let us compute 
\begin{align*}
& g(X_T)\Phi_T- g(X_T)\tilde{\Phi}_T
\\ &= g(X_T)\Gamma_T 
\sum_{j=1}^d \Big(\int_0^T (h^j(X_s)-h^j(X_{\eta(s)}))dY^j_s -\frac{1}{2}\int_0^T ((h^j)^2(X_s) - (h^j)^2(X_{\eta(s)}))ds \Big)
\end{align*}
where 
\begin{align*}
\Gamma_T &= \int_0^1 \Gamma_T(\rho) d\rho, 
\\ \Gamma_T(\rho) &= \exp(\rho \log(\Phi_T) + (1-\rho)\log(\tilde{\Phi}_T)). 
\end{align*}
Applying It\^o's formula to $\zeta(X_s)$ with $\zeta = h^j$ or $(h^j)^2 \in C^2$, we have
\begin{align*}
\zeta(X_s)-\zeta(X_{\eta(s)}) = \int_{\eta(s)}^s \nabla\zeta(X_r)\sigma(X_r)dB_r + \int_{\eta(s)}^s (\gene \zeta)(X_r)dr.
\end{align*}
So the error $E[g(X_T)\Phi_T|\F_T^Y] - E[g(X_T)\tilde{\Phi}_T|\F_T^Y]$ 
can be decomposed into four parts $(E_i)_{1\leq i\leq 4}$: 
\begin{align*}
E_1 &= E\Big[g(X_T)\Gamma_T \sum_{j=1}^d\int_0^T \Big(\int_{\eta(s)}^s \nabla(h^j)(X_r)\sigma(X_r)dB_r\Big) dY^j_s \Big|\F_T^Y\Big] 
\\ E_2 &= E\Big[g(X_T)\Gamma_T \sum_{j=1}^d\int_0^T \Big(\int_{\eta(s)}^s \gene h^j(X_r)dr\Big) dY^j_s \Big|\F_T^Y\Big]
\\ E_3 &= - \frac{1}{2} E\Big[g(X_T)\Gamma_T \sum_{j=1}^d\int_0^T \Big(\int_{\eta(s)}^s \nabla((h^j)^2)(X_r)\sigma(X_r)dB_r\Big) ds \Big|\F_T^Y\Big]
\\ E_4 &= - \frac{1}{2} E\Big[g(X_T)\Gamma_T \sum_{j=1}^d\int_0^T \Big(\int_{\eta(s)}^s \gene (h^j)^2(X_r)dr\Big) ds \Big|\F_T^Y\Big].
\end{align*}
We are going to prove that 
\[
\|E_i\|_p \leq \frac{C(i,p,T)}{n}
\]
for $p \geq 2$ and $1 \leq i \leq 4$. 
The estimation for $E_1$ is the most difficult task since $E_1$ includes both $dB$ and $dY$ parts. 
First, we give the estimates for $E_2$ and $E_4$.

\begin{proposition}\label{prop:E_24}
Under the assumption \rm{(A1)}-\rm{(A4)}, for every $p \geq 1$, 
there exists a constant $C = C(p,T) > 0$ such that 
\[
\|E_2\|_p + \|E_4\|_p \leq \frac{C}{n}.
\]
\end{proposition}
\begin{proof}
By the assumption \rm{(A4)}, it holds 
that 
\[
\|\Gamma_T\|_q \leq \|\Phi_T\|_q + \|\tilde{\Phi}_T\|_q \leq K(q,T) < \infty
\]
for every $q \geq 1$. Thus we have easily 
\begin{align*}
\|E_4\|_p &\leq \|g(X_T)\Gamma_T\|_{2p} 
E\Big[\Big| \sum_{j=1}^d\int_0^T \Big(\int_{\eta(s)}^s \gene (h^j)^2(X_r)dr\Big) ds \Big|^{2p}\Big]^{1/2p}
\\ &\leq \frac{C_1(p,T)}{n} \sum_{j=1}^d E\Big[\sup_{0\leq r \leq T}|\gene (h^j)^2(X_r)|^{2p}\Big]^{1/2p}.
\end{align*}
This gives the estimate $\|E_4\|_p \leq C/n$.

We next turn to prove $\|E_2\|_p \leq C/n$. Using the Cauchy-Schwarz inequality and Burkholder-Davis-Gundy inequality, we have 
\begin{align*}
\|E_2\|_p &\leq \|g(X_T)\Gamma_T\|_{2p} 
E\Big[\Big| \sum_{j=1}^d\int_0^T \Big(\int_{\eta(s)}^s \gene h^j(X_r)dr\Big) dY^j_s \Big|^{2p}\Big]^{1/2p}
\\ &\leq C_2(p,T) \sum_{j=1}^d E\Big[\Big( \int_0^T \Big(\int_{\eta(s)}^s \gene h^j(X_r)dr\Big)^2 ds \Big)^{p}\Big]^{1/2p}.
\end{align*}
We can finally get the estimate 
\begin{align*}
E\Big[\Big( \int_0^T \Big(\int_{\eta(s)}^s \gene h^j(X_r)dr\Big)^2 ds \Big)^{p}\Big]^{1/2p} 
& \leq E\Big[ \sup_{0\leq r \leq T} |(\gene h^j)(X_r)|^{2p} \Big(\int_0^T (s-\eta(s))^2ds \Big)^p \Big]^{1/2p}
\\ & \leq \frac{C_3(p,T)}{n}. 
\end{align*}
\end{proof}

\section{The estimation via infinite dimensional analysis}\label{sec:3}

This section is devoted to the estimates for $E_1$ and $E_3$ defined in previous. 
The Malliavin calculus for Hilbert space valued functionals plays an important role in the estimates. 

\subsection{A brief review of Malliavin calculus and Hilbert space valued martingales}
Let $(\Omega, \F, Q)$ be a $d$-dimensional Wiener space and 
$(B_t)_{0\leq t \leq T}$ be the $d$-dimensional canonical Brownian motion on $(\Omega, \F, Q)$. 
More precisely, $\Omega = C([0,T];\R^d)$, $\F$ is the Borel $\sigma$-field on $\Omega$, and 
$Q$ is the Wiener measure under which the coodinate map $t \mapsto B_t, B \in \Omega$ becomes 
a standard Brownian motion.  

The Malliavin derivative $D : L^2(\Omega) \supset \mathrm{Dom}(D) \rightarrow L^2(\Omega; L^2([0,T]; \R^d))$ 
is defined as the extension of the following closable operator for smooth Wiener functional $F$: 
\[
F = f\Big(\int_0^T h_1(s)dB_s, \dots, \int_0^T h_m(s)dB_s\Big)  
\]
where $f: \R^m \rightarrow \R$ is a polynomial function and $(h_i) \subset L^2([0,T];\R^d)$. Then define
\[
DF := \sum_{i=1}^m (\partial_i f)\Big(\int_0^T h_1(s)dB_s, \dots, \int_0^T h_m(s)dB_s\Big) h_i.
\]
The Skorohod integral $\delta : L^2(\Omega; L^2([0,T]; \R^d)) \supset \mathrm{Dom}(\delta) \rightarrow L^2(\Omega)$ 
is the adjoint operator of $D$. 
Let $K$ be a real separable Hilbert space. We can similarly define $D$ and $\delta$ for $K$-valued Wiener functionals. 
The spaces $\DD^{1,p}(K) \subset L^p(\Omega; K)$ are defined as the Sobolev spaces induced by 
the derivative operator $D$ for $K$-valued Wiener functionals. 
For the details of the precise formulation of Malliavin calculus, we refer to \cite{Shigekawa} and \cite{N09}. 

We prepare some results for the Skorohod integral $\delta$ (cf.\ \cite{N09}).
\begin{lemma}\label{lem:comp_dual}
For $u(\cdot) = \sum_{i=1}^n F_i 1_{[t_i, t_{i+1})}(\cdot) \in L^2([0,T]; \R^d)$ with $F_i \in \DD^{1,2}(\R^d)$, we have  
\[
\delta(u) = \sum_{i=1}^n F_i \cdot (B_{t_{i+1}}-B_{t_i}) - \sum_{i=1}^n \int_{t_i}^{t_{i+1}} \sum_{j=1}^dD_r^{(j)}F_i^{(j)} dr. 
\]
\end{lemma}

\begin{lemma}[Continuity of $\delta$]\label{lem:conti_delta}
Let $p > 1$.
There exists $C  > 0$ such that 
\[
\|\delta(u)\|_p \leq C \|u\|_{\DD^{1,p}(L^2([0,T]; \R^d))}
\]
for every $u \in \DD^{1,p}(L^2([0,T]; \R^d))$
\end{lemma}

We will use a kind of Fubini's theorem below. 
\begin{lemma}\label{lem:fubini}
Let $(u_s)_{0\leq s \leq T} \in L^2([0,T]; \DD^{1,2}(L^2([0,T];\R^d)))$, then 
\begin{equation}\label{eq:fubini}
\int_0^T \delta(u_s(\cdot))ds = \delta\Big( \int_0^T u_s(\cdot) ds \Big) \ \mbox{ a.s.}
\end{equation}
\end{lemma}
\begin{proof}
Let $u_s^k = \sum_{j=1}^{m_k} a_j^k 1_{B_j^k}(s)$ 
with $a_j^k \in \DD^{1,2}(L^2([0,T];\R^d))$ and $B_j^k \in \B([0,T])$ 
such that $u^k \rightarrow u$ in the norm of $L^2([0,T]; \DD^{1,2}(L^2([0,T];\R^d)))$ 
as $k \rightarrow \infty$. 
Clearly we have 
\[
\int_0^T \delta(u_s^k(\cdot))ds = \delta\Big( \int_0^T u_s^k(\cdot) ds \Big). 
\]
It suffices to check the limit of both sides. By taking $L^2$-norm, 
\begin{align*}
\Big\|\int_0^T \delta(u_s^k(\cdot))ds - \int_0^T \delta(u_s(\cdot))ds \Big\|_2^2
& \leq C_1 \int_0^T \| \delta(u_s^k(\cdot) - u_s(\cdot))\|_2^2 ds
\\ & \leq C_2 \int_0^T \| u_s^k(\cdot) - u_s(\cdot)\|_{\DD^{1,2}(L^2([0,T];\R^d))}^2 ds
\end{align*}
and 
\begin{align*}
\Big\|\delta\Big(\int_0^T u_s^k(\cdot)ds \Big) - \delta\Big(\int_0^T u_s(\cdot)ds \Big)\Big\|_2^2
& \leq C_3 \Big\|\int_0^T (u_s^k(\cdot) - u_s(\cdot)) ds \Big\|_{\DD^{1,2}(L^2([0,T];\R^d))}^2
\\ & \leq C_4 \int_0^T \| u_s^k(\cdot) - u_s(\cdot)\|_{\DD^{1,2}(L^2([0,T];\R^d))}^2 ds.
\end{align*}
Thus we obtain the result (\ref{eq:fubini}) as $k \rightarrow \infty$. 
\end{proof}

We can derive the following fundamental inequalities for Hilbert space valued martingales. 
\begin{lemma}\label{lem:fundamental_ine}
Let $M_t$ be a continuous $K$-valued martingale with respect to a filtration 
$(\F_t)$ which satisfies the usual conditions. Then for every $p>0$, there exists positive constants $K_p$, $c_p < C_p$ such that 

Doob's inequality:  
\[
E\Big[\sup_{0\leq t \leq T}|M_t|_K^p\Big] 
\leq K_p E\Big[|M_T|_K^p\Big].
\]

Burkholder-Davis-Gundy's inequality: 
\[
c_p E\Big[\langle M \rangle_T^{p/2}\Big] \leq E\Big[\sup_{0\leq t \leq T}|M_t|_K^p\Big] 
\leq C_p E\Big[\langle M \rangle_T^{p/2}\Big].
\]
\end{lemma}
\begin{proof}
See  e.g.\ \cite[Theorem 3.1]{Shigekawa}.
\end{proof}

\begin{lemma}\label{lem:infi_Ito_rep}
If $F \in L^p(\F_T^B; K)$ for some $p \geq 2$, then there exists 
a unique process $f_s = (f_s^1, \dots, f_s^d)$ such that  
$f_s^i$ are $K$-valued progressively measurable processes satisfying
\[
F = E[F] + \int_0^T f_s dB_s, 
\]
and 
\begin{equation}\label{eq:Lp_estimate}
E\Big[\Big( \int_0^T \sum_{i=1}^d|f_s^i|_K^2 ds \Big)^{p/2}\Big] \leq C_p E[|F|_K^p].
\end{equation}
In particular, if $F \in \DD^{1,2}(\F_T^B; K)$, then we have the so-called Clark-Ocone formula 
\[
f_s(\omega) = E[D_sF|\F_s^B](\omega) \mbox{ a.e. } (s, \omega) \in [0,T] \times \Omega.
\]
\end{lemma}
\begin{proof}
We check only the inequality (\ref{eq:Lp_estimate}) using the inequalities in Lemma \ref{lem:fundamental_ine}: 
\begin{align*}
E\Big[\Big( \int_0^T \sum_{i=1}^d|f_s^i|_K^2 ds\Big)^{p/2} \Big] 
& \leq C_1(p) E\Big[\Big|\int_0^T f_s dB_s\Big|_K^{p/2} \Big] 
\\ & = C_1(p) E[|F - E[F]|_K^p]
\\ & \leq C_2(p) E[|F|_K^p].
\end{align*}
\end{proof}

\subsection{Infinite dimensional It\^o calculus for $E_3$}\label{subsec:dualityofsi}
Let us define two Wiener spaces 
$(\W_B, \mathcal{B}(\W_B), P^{\W_B})$ and $(\W_Y, \mathcal{B}(\W_Y), P^{\W_Y})$ 
on which $(B_t)_{0\leq t \leq T}$ and $(Y_t)_{0\leq t \leq T}$ are canonical Brownian motions respectively. 
From now on we specify 
\[
(\Omega, \F, Q) = (\W_B, \mathcal{B}(\W_B), P^{\W_B}) \times (\W_Y, \mathcal{B}(\W_Y), P^{\W_Y}).
\]
We denote by $E^{\W_B}$ and $E^{\W_Y}$ the expectations under $P^{\W_B}$ and $P^{\W_Y}$ respectively. 
Since $B$ and $Y$ are independent, we notice that 
$E[ \ \cdot \ |\F_T^Y] = E^{\W_B}[ \ \cdot \ ]$. 

We now return to prove $\|E_3\|_p = O(1/n)$. 
The fundamental idea to get the order of convergence is as follows (see also \cite{CKL06}): 
Let $F \in L^2(\W_B \times \W_Y; \R)$ and $\theta_s$ be a $\F_s^B$-adapted process with finite moments. 
We are going to give the error estimates for the type of $E^{\W_B}[F\int_{t_i}^{t_{i+1}} \theta_s dB_s]$. 
Let us consider 
\[
L^2(\W_B \times \W_Y; \R) \cong L^2(\W_B; L^2(\W_Y; \R)).
\]
By Lemma \ref{lem:infi_Ito_rep}, we obtain the representation 
$F = E^{\W_B}[F] + \int_0^T f_s dB_s$; see also Picard's paper \cite[Proposition 1]{P84}. 
Applying this representation to $E^{\W_B}[F\int_{t_i}^{t_{i+1}} \theta_s dB_s]$, 
we obtain a conditional duality formula 
\[
E^{\W_B}\Big[F\int_{t_i}^{t_{i+1}} \theta_s dB_s\Big] = E^{\W_B}\Big[\int_{t_i}^{t_{i+1}} f_s \theta_s ds\Big] \in L^2(\W_Y; \R). 
\]
This means that it is possible to prove the convergence of $O(1/n)$ from the term $\int_{t_i}^{t_{i+1}} \cdot \ ds$ 
if $(f_s)$ has good moment estimates. 

\begin{lemma}\label{lem:crucial}
Let $p \geq 2$ and suppose $F \in L^p(\W_B\times\W_Y; \R)$ has the representation $F = E^{W_B}[F] + \int_0^T f_s dB_s$ 
(in Lemma \ref{lem:infi_Ito_rep}), 
then there exists a constant $C = C(p) > 0$ such that 
\begin{equation}\label{eq:crucial}
E\Big[\Big( \int_0^T |f_s|^2 ds \Big)^{p/2}\Big] \leq C E[|F|^p].
\end{equation}
\end{lemma}
\begin{proof}
Recall that $|\cdot |$ is the norm on $\R^d$.  
We can consider the $L^2(\W_Y; \R)$-valued martingale $\int_0^t f_s dB_s$ 
as the $\R$-valued stochastic integral for the $\R^d$-valued process $f_s$ 
which is progressively measurable with respect to the enlarged filtration 
$\F_s^B \vee \F_T^Y$ on $(\Omega, \F, Q)$ 
through usual approximation arguments 
(see e.g.\ \cite[Lemma 21.2]{C11}). 
We can apply Lemma \ref{lem:fundamental_ine} with $K = \R$ to it.  
\end{proof}

\begin{proposition}\label{prop:E_3}
Under the assumption \rm{(A1)}-\rm{(A4)}, for every $p \geq 1$, 
there exists a constant $C = C(p,T) > 0$ such that 
\[
\|E_3\|_p \leq \frac{C}{n}.
\]
\end{proposition}
\begin{proof}
We prove only the one dimensional case. 
Let $\theta_r = \frac{1}{2}(h^2)'(X_r)\sigma(X_r)$. 
Applying Lemma \ref{lem:infi_Ito_rep} and \ref{lem:crucial} to $g(X_T)\Gamma_T$, we have a representation 
\begin{equation}\label{eq:rep_for_E3}
g(X_T)\Gamma_T = E^{\W_B}[g(X_T)\Gamma_T ] + \int_0^T f_s dB_s. 
\end{equation}
Using It\^o's formula for stochastic integrals 
with respect to $B_t$, we can deduce that 
\[
E^{\W_B}\Big[ E^{\W_B}[g(X_T)\Gamma_T ] \int_0^T\int_{\eta(s)}^s \theta_r dB_r ds \Big] = 0
\]
and 
\begin{align*}
E^{\W_B}\Big[\int_0^T f_s dB_s \int_0^T\int_{\eta(s)}^s \theta_r dB_r ds \Big]
&= \int_0^T E^{\W_B}\Big[\int_0^T f_r dB_r \int_{\eta(s)}^s \theta_r dB_r \Big] ds 
\\ &= \int_0^T E^{\W_B}\Big[ \int_{\eta(s)}^s f_r \theta_r dr \Big] ds
\\ &= \int_0^T\int_{\eta(s)}^s E^{\W_B}[f_r \theta_r] dr ds.
\end{align*}
We notice that 
\[
|E^{\W_B}[f_r\theta_r]| \leq E^{\W_B}[|f_r|^2]^{1/2} \sup_{0\leq r \leq T}E^{\W_B}[|\theta_r|^2]^{1/2}.
\]
Therefore the estimate (\ref{eq:crucial}) in Lemma \ref{lem:crucial} implies 
\begin{align*}
\Big\|E^{\W_B}\Big[\int_0^T\int_{\eta(s)}^s f_r \theta_r dr ds \Big]\Big\|_p^p 
&\leq \sup_{0\leq r \leq T} E^{\W_B}[|\theta_r|^2]^{p/2} 
E^{\W_Y}\Big[\Big(\int_0^T\int_{\eta(s)}^{\eta(s)+T/n} E^{\W_B}[|f_r|^2]^{1/2}dr ds\Big)^p \Big]
\\ & \leq C_1 \Big(\frac{T}{n}\Big)^p E\Big[\Big(\int_0^T |f_r|^2 dr\Big)^{p/2} \Big]
\\ & \leq \frac{C_2}{n^p} \|g(X_T)\Gamma_T\|_p^p
\end{align*}
for some constant $C_2 = C_2(p,T)$. 
\end{proof}

\subsection{Partial Malliavin calculus for $E_1$}\label{subsec:partial}
In order to analyze the $E_1$ term, we again use the representation (\ref{eq:rep_for_E3}) 
\[
g(X_T)\Gamma_T = E^{\W_B}[g(X_T)\Gamma_T ] + \int_0^T f_s dB_s. 
\]
We can then obtain 
\begin{align*}
E_1 &= E\Big[g(X_T)\Gamma_T \sum_{j=1}^d\int_0^T \Big(\int_{\eta(s)}^s \nabla(h^j)(X_r)\sigma(X_r)dB_r\Big) dY^j_s \Big|\F_T^Y\Big] 
\\ &= E\Big[\int_0^T f_s dB_s \sum_{j=1}^d\int_0^T \Big(\int_{\eta(s)}^s \nabla(h^j)(X_r)\sigma(X_r)dB_r\Big) dY^j_s \Big|\F_T^Y\Big].
\end{align*}
We should mention that it is impossible to apply It\^o calculus to the inside of the conditional expectation 
since $(f_s)$ is {\it not} adapted to $\F_s^B \vee \F_s^Y$. 

For this reason, instead of It\^o calculus, we review partial Malliavin calculus introduced in \cite{NZ89}. 
Consider Malliavin calculus for each space of 
$(\W_B, \mathcal{B}(\W_B), P^{\W_B})$ and $(\W_Y, \mathcal{B}(\W_Y), P^{\W_Y})$. 
Let us denote the Sobolev spaces, the Malliavin derivative, and the Skorohod integral 
on  $(\W_B, \mathcal{B}(\W_B), P^{\W_B})$ by 
$\DD_B^{k,p}$, $D_t^B$, $\delta_B$, 
and on $(\W_Y, \mathcal{B}(\W_Y), P^{\W_Y})$ by 
$\DD_Y^{k,p}$, $D_t^Y$, $\delta_Y$. 
We note that $D^B$ and $D^Y$ are naturally extended to $(N+d)$-dimensional Wiener space $(\Omega, \F, Q)$, 
and the pair $(D^B, D^Y)$ coincides with the standard Malliavin derivative $D : \Omega \rightarrow L^2([0,T];\R^{N+d})$ 
in the following sense: Let us consider an orthogonal decomposition 
\[
L^2([0,T];\R^{N+d}) = L_B^2 \oplus L_Y^2
\]
with 
\begin{align*}
L_B^2 &= \{f \in L^2([0,T];\R^{N+d}) : f^{(j)} \equiv 0 \mbox{ for } N < j \leq N+d\} \cong L^2([0,T];\R^{N}),
\\ L_Y^2 &= \{f \in L^2([0,T];\R^{N+d}) : f^{(j)} \equiv 0 \mbox{ for } 1 \leq j \leq N\} \cong L^2([0,T];\R^{d}).
\end{align*}
Let $\Pi_B$ and $\Pi_Y$ be the projections from $L^2([0,T];\R^{N+d})$ to $L_B^2$ and $L_Y^2$ respectively. 
Then we can define $D^B := \Pi_B \circ D$ and $D^Y := \Pi_Y\circ D$ on 
the $(N+d)$-dimensional Wiener space $(\Omega, \F, Q)$. 
This formulation is called the ``partial'' Malliavin calculus (\cite{KS84}, \cite{NZ89}).  

In this section, we realize partial Malliavin calculus using 
a ``Sobolev space valued'' Sobolev space $\DD_B^{1,2}(\DD_Y^{1,2}(\R))$. 
Let us start the detailed formulation. 
Let $K$ be a real separable Hilbert space and 
$G \in L^2(\W_B; K)$. We define by $J_t^B$ the projection so that $G = E^{\W_B}[G] + \int_0^T J_s^B(G)dB_s$. 
In particular, if we take $K = \DD_Y^{1,2}(\R)$ and 
\[
G \in \DD_B^{1,2}(\DD_Y^{1,2}(\R)) \subset L^2(\W_B; \DD_Y^{1,2}(\R)),
\]
we have by the Clark-Ocone formlua 
\begin{equation}\label{eq:OCformula}
J_s^B(G) = E^{\W_B}[ D_s^B G |\F_s^B] \in \DD_Y^{1,2}(\R).
\end{equation}
We note that $\DD_B^{1,2}(\DD_Y^{1,2}(\R)) \not= \DD_{(B,Y)}^{2,2}(\R)$ 
where $\DD_{(B,Y)}^{2,2}(\R)$ is the usual Sobolev space on $\W_B \times \W_Y$. 
One notices that the space $\DD_B^{1,2}(\DD_Y^{1,2}(\R))$ is spanned by 
products of smooth functionals: 
\begin{align*}
F &= f\Big(\int_0^T h_1(s)dB_s, \dots, \int_0^T h_m(s)dB_s\Big)g\Big(\int_0^T \theta_1(s)dY_s, \dots, \int_0^T \theta_\ell(s)dY_s\Big)
\\ & \in L^2(\W_B \times \W_Y; \R) \cong L^2(\W_B; L^2(\W_Y; \R)) 
\end{align*}
with $\{h_i\}_{1\leq i \leq m} \subset L^2([0,T];\R^{N})$, $\{\theta_i\}_{1\leq i \leq \ell} \subset L^2([0,T];\R^{d})$, 
real-valued $C^1$-functions $f$ and $g$. 

Let us first present auxiliary lemma which will be used in later computations. 
\begin{lemma}\label{lem:basicDD}
{\rm (i):} For $G \in L^2(\W_B; \DD_Y^{1,2}(K))$, 
\[
D^Y E^{\W_B}[G] = E^{\W_B}[D^Y G] \ \mbox{ a.s.}
\]
{\rm (ii):} If $\xi \in \DD_B^{1,p}(L^2([0,T];\R^d))$ with some $p \geq 2$, then 
$\int_0^T \xi_s dY_s \in \DD_B^{1,p}(\DD_Y^{1,2}(\R))$ and 
\begin{align*}
D^B\Big( \int_0^T \xi_s dY_s \Big) &= \int_0^T (D^B \xi_s) dY_s, 
\\ D^Y\Big( \int_0^T \xi_s dY_s \Big) &= \xi, 
\\ D^YD^B\Big( \int_0^T \xi_s dY_s \Big) &= D^B D^Y\Big( \int_0^T \xi_s dY_s \Big) = D^B \xi.
\end{align*}
\end{lemma}
\begin{proof}
(i): We choose an approximation sequence $(G_k)$ of the form
$G_k = \sum_{i=1}^m S_i 1_{A_i}$, $S_i \in \DD_Y^{1,2}(K)$ and $A_i \in \mathcal{B}(\W_B)$. 
For each $k$, $G_k$ clearly satisfies the desired equality. Thus we obtain the result 
using the continuity of $D$. 
(ii): This is a version of the proof of \cite[Proposition 1.3.8]{N09}, recall that $D^B(Y_t) = 0$.
\end{proof}

For the proof of the estimate $\|E_1\|_p \leq C/n$, we will take an approximation sequence $(Z_\ell)_\ell \subset \DD_B^{1,2p}(\R)$ 
such that $Z_\ell \rightarrow g(X_T)$ in $L^{2p}(\W_B)$ as $\ell \rightarrow \infty$. 
The following lemma plays a key role for the estimate of $E_1$.

\begin{lemma}\label{lem:key}
Let $p \geq 2$ and $Z \in \DD_B^{1,2p}(\R)$. 
Then under the assumptions \rm{(A2)}-\rm{(A4)}, $Z\Gamma_T(\rho) \in \DD_B^{1,p}(\DD_Y^{1,2}(\R))$. 
Moreover, let $(\theta_s)$ be a $\R^d$-valued continuous $\F_s^B$-progressively measurable process with 
$E[\sup_{0\leq s \leq T}|\theta_s|^4]^{1/4} \leq M$, then there exists a constant $C = C(p,T)$ 
such that 
\begin{equation}\label{eq:keyestimate}
E^{\W_Y}\Big[ \Big( \int_0^T \esssup\displaylimits_{0\leq r \leq T} |E^{\W_B}[ D_r^Y J_s^B(Z\Gamma_T(\rho)) \cdot \theta_s]|^2 ds \Big)^{p/2}\Big] 
\leq M^p C \| Z\Gamma_T(\rho) \|_p^p. 
\end{equation}
\end{lemma}
\begin{proof}
We can check that $X_t \in \cap_{p\geq 1}\DD_B^{1,p}$ under Assumption \rm{(A2)}. 
Using the chain rule of Malliavin derivative, we obtain from Lemma \ref{lem:basicDD} and Assumption \rm{(A4)}
\[
\Gamma_T(\rho, k) := \sum_{l=0}^k \frac{(\log(\Gamma_T(\rho)))^l}{l!} \in \bigcap_{p\geq 1}\DD_B^{1,p}(\DD_Y^{1,2}(\R)).
\]
Thus taking the limit $k \rightarrow \infty$, we can show that 
\[
\Gamma_T(\rho) \in \bigcap_{p\geq 1}\DD_B^{1,p}(\DD_Y^{1,2}(\R)), 
\]
which implies $Z\Gamma_T(\rho) \in \DD_B^{1,p}(\DD_Y^{1,2}(\R))$. 

We now start to prove the desired inequality (\ref{eq:keyestimate}). 
Applying the Clark-Ocone formula (\ref{eq:OCformula}) to $Z\Gamma_T(\rho)$, we deduce that
\[
D_r^Y J_s^B(Z\Gamma_T(\rho)) = D_r^Y E^{\W_B}[ D_s^B (Z\Gamma_T(\rho)) |\F_s^B] =  E^{\W_B}[ D_s^B (Z D_r^Y\Gamma_T(\rho)) |\F_s^B]
\]
almost every $(r,s,\omega) \in [0,T]^2 \times \Omega$. 
We notice that
\[
D_r^Y\Gamma_T(\rho) = \exp(\rho \log(\Phi_T) + (1-\rho)\log(\tilde{\Phi}_T)) (\rho h(X_r)+(1-\rho)h(X_{\eta(r)}))
\]
and then 
\begin{align*}
D_s^B (Z D_r^Y\Gamma_T(\rho)) = &  D_s^B(Z\Gamma_T(\rho)) (\rho h(X_r)+(1-\rho)h(X_{\eta(r)})) 
\\ & + Z\Gamma_T(\rho) D_s^B (\rho h(X_r)+(1-\rho)h(X_{\eta(r)})).
\end{align*}
This formula and the Cauchy-Schwarz inequality for the conditional expectation $E[\cdot|\F_s^B]$ imply 
\begin{align*}
|E^{\W_B}[ D_r^Y J_s^B(Z\Gamma_T(\rho)) \cdot \theta_s]|^2 
\leq & \ 2 E^{\W_B}[| J_s^B(Z\Gamma_T(\rho)) |^2] E^{\W_B}[|(\rho h(X_r)+(1-\rho)h(X_{\eta(r)})) \cdot \theta_s |^2]
\\ & + 2 E^{\W_B}[| Z\Gamma_T(\rho) |^2] E^{\W_B}[|D_s^B (\rho h(X_r)+(1-\rho)h(X_{\eta(r)}))\cdot \theta_s|^2].
\end{align*}
We refer for the reader to the basic estimate (\cite{N09}): for any $q \geq 1$, 
\begin{equation}\label{eq:proof_2}
E^{\W_B}\Big[\sup_{0\leq t \leq T}|X_t|^q\Big] 
+ \sup_{0\leq s \leq T}E^{\W_B}\Big[\sup_{0\leq t \leq T}|D_s^B X_t|^q\Big] \leq C_1(q,T) < \infty.
\end{equation}
The above inequality allows us to show that
\[
|E^{\W_B}[ D_r^Y J_s^B(Z\Gamma_T(\rho)) \cdot \theta_s]|^2 
\leq C_2(p,T) ( E^{\W_B}[| J_s^B(Z\Gamma_T(\rho)) |^2] + E^{\W_B}[| Z\Gamma_T(\rho) |^2] ).
\]
We can show by Jensen's inequality and Lemma \ref{lem:crucial} that 
\begin{align*}
E^{\W_Y}\Big[ \Big( \int_0^T E^{\W_B}[ |J_s^B(Z\Gamma_T(\rho))|^2] ds \Big)^{p/2}\Big] 
& \leq E\Big[ \Big( \int_0^T |J_s^B(Z\Gamma_T(\rho))|^2 ds \Big)^{p/2}\Big]
\\ & \leq C_3(p) E[|Z\Gamma_T(\rho)|^p].
\end{align*}
Using these inequalities, we obtain the constant $C$ in the assertion. 
\end{proof}

We now finish the proof of the main theorem. 
\begin{proposition}
Let the assumptions \rm{(A1)}-\rm{(A4)} hold. Then for every $p \geq 2$, 
there exists a constant $C = C(p,T) > 0$ such that 
\[
\| E_1 \|_p \leq \frac{C}{n}.
\]
\end{proposition}
\begin{proof}
We first define 
\[
E_1(\rho) := E\Big[g(X_T)\Gamma_T(\rho) \sum_{j=1}^d\int_0^T \Big(\int_{\eta(s)}^s \nabla(h^j)(X_r)\sigma(X_r)dB_r\Big) dY^j_s \Big|\F_T^Y\Big] 
\]
and then 
\[
\|E_1\|_p \leq \int_0^1 \|E_1(\rho)\|_p d\rho \leq \sup_{0\leq \rho \leq 1} \|E_1(\rho)\|_p. 
\]
So it suffices to give an estimate for $\|E_1(\rho)\|_p$. 

Let us define for $Z \in \DD_B^{1,2p}(\R)$ 
\[
E_1(\rho, Z) := E\Big[Z\Gamma_T(\rho) \sum_{j=1}^d\int_0^T \Big(\int_{\eta(s)}^s \nabla(h^j)(X_r)\sigma(X_r)dB_r\Big) dY^j_s \Big|\F_T^Y\Big].
\]
We shall show that 
\begin{equation}\label{eq:finish}
\|E_1(\rho, Z)\|_p \leq \frac{C}{n} \|Z \Gamma_T(\rho)\|_p, 
\end{equation}
and then taking an approximation sequence $(Z_\ell)_\ell \subset \DD_B^{1,2p}(\R)$ such that $Z_\ell \rightarrow g(X_T)$ in $L^{2p}$, 
we have
\[
\|E_1(\rho)\|_p \leq \frac{C}{n}\|g(X_T)\Gamma_T(\rho)\|_p \leq \frac{\tilde{C}(p,T)}{n},  
\]
which is what we want to prove. 

For notational simplicity, we prove (\ref{eq:finish}) only the case where $B$ and $Y$ are one dimensional Brownian motions. 
Let $\theta_r = (h)'(X_r)\sigma(X_r)$. By It\^o's formula, 
\begin{align*}
\int_0^T\int_{\eta(s)}^s \theta_r dB_r dY_s
= \sum_{i=0}^{n-1} \Big( \Big(\int_{t_i}^{t_{i+1}} \theta_s dB_s\Big) (Y_{t_{i+1}}-Y_{t_i}) 
- \int_{t_i}^{t_{i+1}}(Y_s-Y_{t_i})\theta_s dB_s \Big).
\end{align*}
Set $f_s = f_s(\rho, Z) := J_s^B(Z\Gamma_T(\rho))$. 
We can deduce that 
\[
E^{\W_B}\Big[ E^{\W_B}[Z\Gamma_T(\rho)] \int_0^T\int_{\eta(s)}^s \theta_r dB_r dY_s \Big] = 0
\]
and 
\begin{align*}
&E^{\W_B}\Big[\int_0^T f_s dB_s \int_0^T\int_{\eta(s)}^s \theta_r dB_r dY_s \Big]
\\ &= E^{\W_B}\Big[\int_0^T f_s dB_s \sum_{i=0}^{n-1} \Big( \Big(\int_{t_i}^{t_{i+1}} \theta_s dB_s\Big) (Y_{t_{i+1}}-Y_{t_i}) 
- \int_{t_i}^{t_{i+1}}(Y_s-Y_{t_i})\theta_s dB_s \Big) \Big]
\\ &= E^{\W_B}\Big[\sum_{i=0}^{n-1} \Big( \Big(\int_{t_i}^{t_{i+1}} f_s \theta_s ds \Big) (Y_{t_{i+1}}-Y_{t_i}) 
- \int_{t_i}^{t_{i+1}}(Y_s-Y_{t_i})f_s\theta_s ds \Big) \Big]
\\ &= \sum_{i=0}^{n-1} \Big( \Big(\int_{t_i}^{t_{i+1}} E^{\W_B}[f_s \theta_s] ds \Big) (Y_{t_{i+1}}-Y_{t_i}) 
- \int_{t_i}^{t_{i+1}}(Y_s-Y_{t_i})E^{\W_B}[f_s\theta_s] ds \Big).
\end{align*}
By using Lemma \ref{lem:comp_dual} and the fact that $D^Y E^{\W_B}[\cdot] = E^{\W_B} [D^Y \cdot]$ in Lemma \ref{lem:basicDD}, 
it holds that 
\begin{align*}
& \sum_{i=0}^{n-1} \Big(\int_{t_i}^{t_{i+1}} E^{\W_B}[f_s \theta_s] ds \Big) (Y_{t_{i+1}}-Y_{t_i}) 
\\ & = \delta_Y\Big( \sum_{i=0}^{n-1} \Big(\int_{t_i}^{t_{i+1}} E^{\W_B}[f_s \theta_s] ds \Big) 1_{[t_i,t_{i+1})}(\cdot)\Big)
+ \sum_{i=0}^{n-1} \int_{t_i}^{t_{i+1}}\int_{t_i}^{t_{i+1}} E^{\W_B}[(D_r^Y f_s) \theta_s] ds dr, 
\end{align*}
and 
\begin{align*}
& \sum_{i=0}^{n-1} \int_{t_i}^{t_{i+1}}(Y_s-Y_{t_i})E^{\W_B}[f_s \theta_s] ds 
\\ & = \int_0^T \Big(\delta_Y(E^{\W_B}[f_s\theta_s] 1_{[\eta(s), s)}(\cdot)) + \int_{\eta(s)}^s E^{\W_B}[(D_r^Y f_s) \theta_s] dr \Big)ds
\\ & = \delta_Y\Big( \sum_{i=0}^{n-1} \Big(\int_{\cdot}^{t_{i+1}} E^{\W_B}[f_s \theta_s] ds \Big) 1_{[t_i,t_{i+1})}(\cdot) \Big) 
+ \int_0^T \int_{\eta(s)}^s E^{\W_B}[(D_r^Y f_s) \theta_s] dr ds.
\end{align*}
Here we used Lemma \ref{lem:fubini} in the second equality. 
Consequently we derive the formula 
\begin{align*}
& E^{\W_B}\Big[\int_0^T f_s dB_s \int_0^T\int_{\eta(s)}^s \theta_r dB_r dY_s \Big] 
\\ &= \delta_Y\Big( \sum_{i=0}^{n-1} \Big(\int_{t_i}^{\cdot} E^{\W_B}[f_s \theta_s] ds \Big) 1_{[t_i,t_{i+1})}(\cdot)\Big)
+ \int_{0}^{T}\int_{\eta(r)}^{r} E^{\W_B}[(D_r^Y f_s) \theta_s] ds dr. 
\end{align*}
Using the above formula and Lemma \ref{lem:conti_delta}, we finally get the estimate 
\begin{align*}
\|E_1(\rho,Z)\|_p^p &\leq \frac{C_1}{n^p}E\Big[ \Big( \int_0^T |f_s|^2 ds \Big)^{p/2}\Big]  
+ \frac{C_2}{n^{p/2}} 
E^{\W_Y}\Big[\Big( \sum_{i=0}^{n-1} \int_{t_i}^{t_{i+1}} \int_{t_i}^{t_{i+1}} |E^{\W_B}[(D_r^Yf_s)\theta_s]|^2 ds dr \Big)^{p/2}\Big]
\\ & \leq \frac{C_3}{n^p} \|Z\Gamma_T(\rho)\|_p^p 
+ \frac{C_4}{n^p} E^{\W_Y}\Big[ \Big( \int_0^T \esssup\displaylimits_{0\leq r \leq T} |E^{\W_B}[ (D_r^Y f_s) \theta_s]|^2 ds \Big)^{p/2}\Big] 
\end{align*}
Applying Lemma \ref{lem:key} to the last term, we obtain the result (\ref{eq:finish}). This finishes the proof. 
\end{proof}

\section{Conclusion and some remarks on further research}
The generalization discussed in the present paper consists of two parts. 
The first one is to determine the rate of convergence even if $g$ is irregular, 
and the analysis relies on 
the duality of stochastic integrals in Section \ref{subsec:dualityofsi} 
and a sharp estimate via partial Malliavin calculus in Section \ref{subsec:partial}.
The second one is the estimate by $L^p$-norm with $p > 2$, 
which is derived from the computation of the Skorohod integral 
and its continuity by means of Lemma \ref{lem:conti_delta}. 

We finally remark three problems which should be take into account in future research.

i) 
The author expects that the method of proof works as well 
if $X$ is the solution of a L\'evy-driven stochastic differential equation (independent of $W$). 
Of course, we need several techniques on Wiener-Poisson space 
such as the Clark-Ocone formula and its moment estimates. 
In addition to the duality of the form 
\[
E\Big[ \int_0^T f_s dB_s \int_{t_i}^{t_{i+1}} \theta_s dB_s \Big] 
= E\Big[ \int_{t_i}^{t_{i+1}} f_s \theta_s ds \Big], 
\]
we will also use the duality for a Poisson random measure $N(dx, dt)$ of the form
\[
E\Big[ \int_0^T f_s(x) \tilde{N}(dx, ds) \int_{t_i}^{t_{i+1}} \theta_s(x) \tilde{N}(dx, ds) \Big] 
= E\Big[ \int_{t_i}^{t_{i+1}} \int_{\R^N} f_s(x) \theta_s(x) \nu(x)dx ds \Big]
\]
where $\tilde{N}$ is the compensated  Poisson random measure and $\nu$ is the L\'evy measure associated with $N(dx,ds)$.
The detailed discussion is left for future work. 

ii) 
If $B$ and $W$ are not independent (more generally, $X$ depends on $W$), 
we cannot apply the procedure of our proof to the error estimates. 
To begin with, the rate of convergence is not clear ($n^{-1/2}$ or $n^{-1}$) in that case. 
Similarly, the case where the coefficient $h$ depends on $Y$ is also quite complicated situation 
to determine the rate of convergence. 

iii) Another subject of interest in this field 
is an asymptotic limit (central limit theorem) with rate $1/n^\alpha$ by means of 
\[
n^\alpha \Big( E[g(X_T)\Phi_T|\F_T^Y] - E[g(X_T)\tilde{\Phi}_T|\F_T^Y] \Big) \rightarrow G \not= 0 \ \mbox{ in law, } 
\]
which implies the optimal rate of convergence of the conditional expectation. 
The result in Theorem \ref{thm:main} is not sufficient for this purpose 
since we merely can take $\alpha = 1-\epsilon$ (any $\epsilon >0$) with $G=0$. 

\section*{Acknowledgement}
The author would like to thank Professor Masatoshi Fujisaki for giving him 
the opportunity to study nonlinear filtering and motivating this research through helpful discussion. 
This work was supported by JSPS KAKENHI Grant Number 12J03138.


\end{document}